\documentclass{article}

\usepackage[utf8]{inputenc}
\usepackage{graphicx}
\usepackage{amsfonts}
\usepackage{amsmath}
\usepackage{amsthm}
\usepackage{amssymb}
\usepackage{pgfplots}
\usepackage{graphicx}
\usepackage{float}
\usepackage{upgreek}
\usepackage{url}
\usepackage{siunitx}
\usepackage{booktabs}
\usepackage{gensymb}
\usepackage{enumitem}
\usepackage{multicol}
\usepackage{mathrsfs}
\usepackage{algpseudocode}
\usepackage{adjustbox}
\usepackage{algorithm}
\usepackage{mathtools}
\usepackage{tikz-cd}
\usepackage[font=small,labelfont=bf]{caption}
\usepackage{wrapfig}
\usepackage[margin=1.32in]{geometry}
\usepackage[
backend=biber,
style=ieee,
citestyle=numeric-comp
]{biblatex}
\usepackage{titling}
\usepackage{hyperref}
\usepackage{xurl}
\hypersetup{
    colorlinks=true,
    linkcolor=blue,
    citecolor=red,
    breaklinks=true
}

\title{Existence of Uncertainty Minimizers for the \\Continuous Wavelet Transform}
\author{Simon Halvdansson\thanks{Department of Mathematics, NTNU Norwegian University of Science and Technology}, Jan-Fredrik Olsen\thanks{Centre for Mathematical Sciences, Lund University}, Nir Sochen\thanks{School of Mathematical Sciences, Tel Aviv University}, and Ron Levie\thanks{Faculty of Mathematics, Technion – Israel Institute of Technology}}
\thanksmarkseries{arabic}
\date{}

\theoremstyle{plain}
\newtheorem{theorem}{Theorem}[section]
\newtheorem*{theorem*}{Theorem}
\newtheorem{lemma}[theorem]{Lemma}
\newtheorem{proposition}[theorem]{Proposition}

\theoremstyle{definition}
\newtheorem{definition}[theorem]{Definition}

\theoremstyle{remark}

\newcommand{\C}{\mathbb{C}}
\newcommand{\R}{\mathbb{R}}

\newcommand{\N}{\mathbb{N}}

\makeatletter
\newcommand{\vast}{\bBigg@{4}}
\newcommand{\Vast}{\bBigg@{5}}
\makeatother

\def\XXint#1#2#3{{\setbox0=\hbox{$#1{#2#3}{\int}$ }
		\vcenter{\hbox{$#2#3$ }}\kern-.6\wd0}}

\newcounter{RonCounter}

\newcounter{SimonCounter}

\newcounter{OlsenCounter}

\begin{document}
	
	\maketitle
	\vspace{-1cm}
	\begin{abstract}
	Continuous wavelet design is the endeavor to construct mother wavelets with desirable properties for the continuous wavelet transform (CWT). One class of methods for choosing a mother wavelet involves minimizing a functional, called the wavelet uncertainty functional.
	Recently, two new wavelet uncertainty functionals were derived from theoretical foundations. In both approaches, the uncertainty of a mother wavelet describes its concentration, or accuracy, as a time-scale probe. While an uncertainty minimizing mother wavelet can be proven to have desirable localization properties, the existence of such a minimizer was never studied.  In this paper, we prove the existence of minimizers for the two uncertainty functionals.
	\end{abstract}
	
	
	\section{Introduction}
	For   $s, f \in H^2(\mathbb{R}) = \big\{ q \in L^2(\mathbb{R}) \,:\,\operatorname{supp}(\hat{q}) \subset \mathbb{R}^+ \big\}$, where $\R^+$ denotes the positive real numbers,  and $(\alpha, \beta) \in \mathbb{R}^2$,  we consider  the continuous wavelet transform (CWT) on the form
	\begin{align}
	\label{eq:Wfs}
		W_f[s](\alpha, \beta) &= \big\langle s, \pi(\alpha, \beta) f\big\rangle_{L^2(\mathbb{R})}.
	\end{align}	
	Here, $s$ is called the \textit{signal}, $f$   the \textit{mother wavelet}, and  $\pi(\alpha,\beta)$  is the wavelet representation 
	\begin{align}
	\label{eq:Wfs2}
		\pi(\alpha, \beta)f(t) &= e^{-\alpha/2}f\left(\frac{t-\beta}{e^\alpha}\right).
	\end{align}
	That is, the signal $s$ is analyzed by taking the inner product with dilations and translations of the mother wavelet. We can consider these operations separately, by writing  $\pi(\alpha,\beta)=\pi_2(\beta)\pi_1(\alpha)$, where $\pi_1(\alpha)f(t) = e^{-\alpha/2} f\left(\frac{t}{e^{\alpha}}\right)$ represents the dilations of $f$ and  $\pi_2(\beta) f(t) = f(t - \beta)$ translations.
	Defined in this way, the continuous wavelet transform is a bounded operator of the form
	$$W_f:H^2(\mathbb{R})\rightarrow L^2(\mathbb{R}^2;d\mu(\alpha,\beta)),$$
	 where
	 $d\mu(\alpha,\beta)=e^{-\alpha}\, d\alpha d\beta$ is the Haar measure of the affine group  \cite{grossman, daubechiesTen}.

	 The wavelet transform of functions in $L^2(\mathbb{R})$ can be treated by analyzing the positive and negative frequency supports separately. The mother wavelet $f$ is required to  be \emph{admissible}, namely, it must satisfy
	 \begin{equation}
	     \label{eq:admis}
	     \int_0^\infty \frac{|\hat{f}(\omega)|^2}{\omega} d\omega < \infty.
	 \end{equation} This guarantees that the wavelet transform is invertible \cite{mallat_tour}.
    
    In \eqref{eq:Wfs}, the mother wavelet $f$ is a free parameter, and the endeavor to construct mother wavelets with desirable properties is called \emph{wavelet design}.
	In this paper, we focus on wavelet design methods based on minimizing \emph{uncertainty functionals}.
	The uncertainty of a mother wavelet is typically interpreted as its \textit{sharpness}  as a time-scale probe. 
	In the short-time Fourier transform (STFT), the sharpness of a window function is defined as its Heisenberg uncertainty, and hence optimal windows are Gaussian functions \cite{grochenig_book}. 
	One classical attempt to generalize this approach to wavelet analysis is to generalize the Heisenberg uncertainty principle by taking infinitesimal group generators of $\pi$ as localization operators \cite{dahlke1995}. %While the group generator approach for defining localization is fruitful for the STFT, it only leads to trivial results   in the wavelet context \cite{levie2013, maass2010}. \ron{I would not say that the results are trivial. Perhaps counter-intuitive? rather strange?} \olsen{Perhaps one should say that the situation is "different" on the wavelet-side, and then formulate a sentence or two about why. That is, briefly indicate the points of \cite{levie2013, maass2010}. I guess this is one of the main selling points of why the results of this paper are interesting.}
	%\ron{While the group generator approach for defining localization is fruitful for the STFT, this is not the case for other transforms, like the CWT \cite{levie2013, maass2010}. In fact, when considering the one parameter time translation representation $\pi_2(\beta) f(t) = f(t - \beta)$, a sharp time probe $f$ should intuitively have a small time spread. However, the infinitesimal generator of $\pi_1(\beta)$ is $i\frac{\partial}{\partial t}$, which is the \emph{frequency} localization operator. A sharp probe in the sense of $i\frac{\partial}{\partial t}$ has a small \emph{frequency} spread, which is in some sense the opposite of what we want to model.}
	While the group generator approach for defining localization is fruitful for the STFT, this is not the case for other transforms, like the CWT, as is explained in \cite{levie2013, maass2010}. 
	
	An alternative approach for defining a wavelet uncertainty, based on the concept of \emph{observables}, was proposed and investigated  in \cite{levie2013, levie2020, levie2021}.
	Observables are  localization operators that enable us to define uncertainty functionals  that measure the localization of mother wavelets $f$ in time and scale.
	The approach was shown to be meaningful in the sense that the uncertainty of a mother wavelet is directly linked to the sparsity, or sharpness, in the corresponding coefficient space.

	 Two observable-based uncertainty functionals were  proposed in \cite{levie2020} and \cite{levie2021}. However, the existence of  minimizers of these uncertainty functionals was not proved.  
	 In this paper, we prove the existence of a wavelet uncertainty minimizer in both cases.

	We note that alternative approaches to wavelet design include the method introduced by Daubechies to construct compactly supported wavelets with vanishing moments \cite{daubechiesTen}, and adaptive methods which aim to maximize the correlation between the mother wavelet and the signal  \cite{chapa}.
	
	\section{Wavelet uncertainty functionals}
	In this section, we recall the observables approach to wavelet uncertainty functionals, and the two wavelet uncertainties introduced in \cite{levie2013, levie2021}.
	
	\subsection{The wavelet transform in the frequency domain}
	
	Wavelet uncertainty functionals are represented more naturally in the frequency domain than in the time domain. Hence, in this paper, the default space in which we work is the frequency domain. Accordingly, we denote mother wavelets and signals in the signal space $L^2(\R^+) = \{f\in L^2(\R)\ |\ \operatorname{supp}(f) \subset \mathbb{R}^+  \}$ by $f$, without a hat notation. We denote signals in the time domain by $\check{f}$. The wavelet representation \eqref{eq:Wfs2}  is now given by 
	$$
	\pi(\alpha, \beta) f(\omega) = e^{-2\pi i \omega \beta} e^{\alpha/2} f(e^\alpha \omega).
	$$
	
	\subsection{Wavelet localization operators}
    The approach for defining uncertainty functionals is based on taking \emph{observables} as localization operators. Inspired by quantum mechanics, an observable is a symmetric operator \cite{HAPS,quantum_measure}. In the signal processing context, observable are interpreted as entities that measure some underlying physical quantities of signals \cite{levie2020}. For example, the multiplication operator
    \begin{equation}
    \label{eq:Tx}
       \check{T}_x \check{f}(t)=t\check{f}(t) 
    \end{equation} measures localisation in \emph{time} of  signals $\check{f}\in H^2(\R)$. That is, when treating $|\check{f}(t)|^2$ as the density of the signal $\check{f}$ at time $t$, the mean time of $\check{f}$ is defined as the center of mass
    $$
    \big\langle \check{T}_x \check{f}, \check{f} \big\rangle = \int_\R t|\check{f}(t)|^2\, dt.
    $$
    
    The following definition extends the above discussion for general observables, and  explains what is meant by the spread of a signal about an observed quantity. 
    
	\begin{definition}
	    Let $T$ be a symmetric operator on a Hilbert space $\mathcal{H}$  and $f$ a normalized vector in the domain  $\operatorname{Dom}(T)$ of $T$. The \emph{expected value} and \emph{variance} of $f$ with respect to $T$ are defined, respectively, as
	    \begin{align*}
    		e_f(T) &= \big\langle Tf, f \big\rangle,\\
    		v_f(T) &= \big\Vert (T - e_f(T)) f \big\Vert^2.
    	\end{align*}
	\end{definition}
	When we want to emphasise the space in which the inner product and norm are defined, we add a superscript to $e$ and $v$, e.g., $e_f^{\mathcal{H}}(T)$. 

	Since the wavelet transform represents signals in the time-scale space $L^2(\mathbb{R}^2; d\mu(\alpha,\beta))$, the wavelet observables are the \emph{time} and the \emph{scale} observables. The time observable (\ref{eq:Tx}) multiplies by the variable of the time space $H^2(\mathbb{R})$, in which $\pi_2(\beta)$ is represented as translation.	Similarly, there is a \emph{scale space} where $\pi_1(\alpha)$ is represented as a translation, and in which the scale observable is defined as a multiplication operator. % \olsen{I think it is hard for the uninitiated reader to understand what is meant by scale space at this point, or why $\pi_2(\beta)$ is a translation in this space. My feeling is that it could be helpful to first explain the story for the time observable to the reader (which can be short since everyone gets it), and then use the word "similarly" as a bridge to give the corresponding, short, explanation for the scale observable...}\simon{I changed the text a bit here to make the explanation a bit more analagous.} \ron{I agree with the change}
	The mapping between the time space and the scale space is the \emph{scale transform} $U$, defined as
	\begin{align*}
	    U: L^2(\R^+)\rightarrow L^2(\mathbb{R}), \quad U\{f\}(\sigma) = e^{-\sigma/2} f(e^{-\sigma}) =: \tilde{f}(\sigma)
	\end{align*}
	
	The motivation behind the above construction is that, in order to measure the quantity which is transformed by $\pi_j$, $j=1,2$, we first represent $\pi_j$ as a translation operator, and then we treat the translated variable as the physical quantity corresponding to $\pi_j$. 
	
	Given an observable $T$, we denote its form in the time and scale spaces by $\check{T} = \mathcal{F} T \mathcal{F}^*$ and $\tilde{T} = U^* T U$, where $\mathcal{F}$ is the Fourier transform. To formally define the time and scale observables
	we denote by
	\begin{equation}
	    \label{eq:mult}
	    Y : f(y) \mapsto yf(y)
	\end{equation}
	 the general multiplication operator in $L^2(\R)$ with the domain $\operatorname{Dom}(Y) = \big\{ f \in L^2(\R) : y\mapsto yf(y) \in L^2(\R) \big\}$.

	\begin{definition}
	    The \emph{time} and \emph{scale observables} $T_x, T_\sigma$ are defined in the signal space $L^2(\R^+)$ as
	    \begin{align*}
	        T_x &= \mathcal{F} Y \mathcal{F}^*,\\
	        T_\sigma  &= U^*YU.
	    \end{align*}
	\end{definition}
	
	\begin{proposition}[\cite{levie2013}]\label{prop:frequency_observables}
    The signal space $L^2(\R^+)$ is invariant under the observables $T_x$ and $T_{\sigma}$. 
	    Moreover,
	    \begin{align*}
	        T_x f(\omega) &= i\frac{\partial}{\partial \omega} f(\omega),\\
	        T_\sigma f(\omega) &= -\ln(\omega) f(\omega).
	    \end{align*}
	     The domains of $T_x$ and $T_{\sigma}$ are the sets of $f \in L^2(\R^+)$ such that $f$ is absolutely continuous in $[a, b]$ for every $-\infty<a < b<\infty$, and  $\omega \mapsto -\ln(\omega) f(\omega) \in L^2(\R^+)$ respectively.
	\end{proposition}
	Note that every $f$ in the domain of $T_x $ must satisfy $f(0)=0$. Indeed, $f$ is continuous and $f(\omega)=0$ for every $\omega<0$. Moreover, $L^2(\R^+)$ is invariant under $T_x$, so $H^2(\R)$ is invariant under $Y$. 
	
	The scale transform is canonical in the sense that it transforms dilations $\pi_1$ to translations.  It can therefore  be verified that the following \emph{canonical commutation relations} \cite{levie2020} hold
	\begin{align*}
	\pi_1(\alpha)^* T_\sigma \pi_1(\alpha) = T_\sigma + \alpha, \\
	\pi_2(\beta)^* T_x \pi_2(\beta) = T_x + \beta.
	\end{align*}
	Moreover, we have 
	\begin{align*}
	    e_{\pi(\alpha, 0) f}(T_\sigma) &= e_f(T_\sigma) + \alpha,\\
	    e_{\pi(0, \beta) f}(T_x) &= e_f(T_x) + \beta,\\
	    v_{\pi(\alpha, 0) f}(T_\sigma) &= v_f(T_\sigma),\\
	    v_{\pi(0, \beta) f}(T_x) &= v_f(T_x).
	\end{align*}
	Since admissible wavelets form a subset of $L^2(\R^+)$, we make the distinction between the \emph{signal space} $L^2(\R^+)$, which we now denote by $\mathcal{S}$, and the \emph{window space} $\mathcal{W}$ \cite{levie2017}. The inner product in the window space is defined according to the admissibility condition \eqref{eq:admis} as
	\begin{align*}
		\Big\langle f_1, f_2\Big\rangle_\mathcal{W} = \int_0^\infty f_1(\omega) \overline{f_2(\omega)}\frac{1}{\omega} d\omega.
	\end{align*}
	The window space $\mathcal{W}$ is defined to be the completion to a Hilbert space of $L^2(\R^+)$ with the inner product $\big\langle f_1, f_2\big\rangle_\mathcal{W}$. Concretely, $\mathcal{W}$ is the weighted Lebesgue space $L^2(\R^+, \frac{1}{\omega}d\omega)$. We call $e_f^\mathcal{S}(T_x)$ the \emph{signal-expected time},  and $e_f^\mathcal{W}(T_x)$ the \emph{window-expected time}, and similarly use the terms signal-expected scale, window-time variance, etc.    
	
	\subsection{Signal space uncertainty}
    The signal space uncertainty of $f\in L^2(\R^+)$, introduced in \cite{levie2013}, is based on a combination of the time and scale variances of $f$.
	\begin{definition}\label{def:signal_uncertainty}
	    The \emph{signal space uncertainty} $\mathcal{L}_{\rm S}$ is defined as
	    $$
    	\mathcal{L}_{\rm S}(f) = e^{-2e_{\frac{f}{\Vert f \Vert_\mathcal{S}}}^\mathcal{S}(T_\sigma)}v_{\frac{f}{\Vert f \Vert_\mathcal{S}}}^\mathcal{S}(T_x) + v_{\frac{f}{\Vert f \Vert_\mathcal{S}}}^\mathcal{S}(T_\sigma)
    	$$
    	on the domain $\operatorname{Dom}(\mathcal{L}_{\rm S}) = \operatorname{Dom}(T_x) \cap \operatorname{Dom}(T_\sigma)$.
	\end{definition}
	The exponential term in the definition guarantees that $\mathcal{L}_{\rm S}$ is invariant under $\pi(\alpha,\beta)$. This is a consequence of the following lemma.
	\begin{lemma}[\cite{levie2020}]\label{lemma:signal_invariance}
	    The signal space uncertainty $\mathcal{L}_{\rm S}$ is invariant under $\pi(\alpha,\beta)$ and linear scalings in the sense that
	    $$
	    \mathcal{L}_{\rm S}(f) = \mathcal{L}_{\rm S}(c\pi(\alpha, \beta) f)\qquad\text{ for all }c \in \C\text{ and } (\alpha, \beta)\in\R^2.
	    $$
	    In particular, for any $f\in L^2(\R^+)$, the normalized signal $f_{\text{N}} = \frac{1}{\Vert f \Vert_\mathcal{S}} \pi\big(e_{\frac{f}{\Vert f \Vert_\mathcal{S}}}^\mathcal{S}(T_x), e_{\frac{f}{\Vert f \Vert_\mathcal{S}}}^\mathcal{S}(T_\sigma)\big)^{-1} f$ satisfies $\Vert f_{\text{N}} \Vert_\mathcal{S} = 1$, $e_{f_{\text{N}}}^\mathcal{S}(T_x) = e_{f_{\text{N}}}^\mathcal{S}(T_\sigma) = 0$ and
	    \begin{align}\label{eq:simplified_signal_uncertainty}
	        \mathcal{L}_{\rm S}(f) = \mathcal{L}_{\rm S}(f_{\text{N}}) = \big\Vert T_x f_{\text{N}}\big\Vert_\mathcal{S}^2 + \big\Vert T_\sigma f_{\text{N}} \big\Vert_\mathcal{S}^2.
	    \end{align}
	\end{lemma}
	The following proposition shows that  functions in the domain of $\mathcal{L}_{\rm S}$ are admissible.
	\begin{proposition}\label{prop:dom_S_admissible}
	    Any element of $\operatorname{Dom}(\mathcal{L}_{\rm S})$ is admissible, i.e., $\operatorname{Dom}(\mathcal{L}_{\rm S}) \subset \mathcal{W} \cap \mathcal{S}$.
	\end{proposition}
	\begin{proof}
	    Let $f \in \operatorname{Dom}(\mathcal{L}_{\rm S})$. Then $f$ is continuous and zero for negative $\omega$, so $f(0) = 0$. Moreover, since $f$ is in the domain of $T_x$, it holds that $f' \in L^2(\R^+)$. It then follows that
	    $$
	    \big|f(\omega)\big| \leq \int_0^\omega \big|f'(\xi)\big|\, d\xi \leq \sqrt{\big\Vert f' \big\Vert_{L^2}} \sqrt{\omega}
	    $$
	    by the Cauchy-Schwarz inequality. Hence
	    $$
	    \int_0^\infty \frac{|f(\omega)|^2}{\omega} d\omega \leq \int_0^1 \big\Vert f' \big\Vert_{L^2} \,d\omega + \big\Vert f \big\Vert_{L^2} < \infty.
	    $$
	\end{proof}
	
	\subsection{Phase space uncertainty}
	The \emph{phase space uncertainty}, introduced in \cite{levie2021}, is a way to model the spread of the 2D \emph{ambiguity function}
	$$
	K_f(\alpha, \beta) = W_f[f](\alpha, \beta) = \big\langle f, \pi(\alpha, \beta)f \big\rangle_{\mathcal{S}}.
	$$
	The ambiguity function determines the amount of `blurriness'  of the output of the wavelet transform in the coefficient space $L^2(\mathbb{R}^2;d\mu(\alpha,\beta))$. Indeed, $K_f$ is the reproducing kernel of $W_f[\mathcal{S}]$ \cite{fuhr05, levie2017}, meaning that
	\begin{align*}
		Q \in W_f[\mathcal{S}] \implies Q = Q * K_f.
	\end{align*}
	Hence, the spread of $K_f$ poses an upper bound on the resolution of the wavelet coefficient space.
	
	The phase space uncertainty is based on the variance of the \emph{phase space scale} and \textit{phase space time observables}. These are defined, respectively,    for $F : \R^2 \to \C$, by 
	\begin{align*}
	AF(\alpha, \beta) &= \alpha F(\alpha, \beta), \\
	%$$
	%and the \emph{phase space time observable} $B$,
	%$$
	BF(\alpha, \beta) &= \beta F(\alpha, \beta).
	\end{align*}
	These operators are self-adjoint on their domains 
	\begin{align*}
		\operatorname{Dom}(A) &= \big\{ F \in L^2(\R^2, d\mu(\alpha, \beta)) \ : \  (\alpha, \beta) \mapsto \alpha F(\alpha, \beta) \text{ is in }L^2(\R^2, d\mu(\alpha, \beta)) \big\}, \\
		\operatorname{Dom}(B) &= \big\{ F \in L^2(\R^2, d\mu(\alpha, \beta)) \ : \ (\alpha, \beta) \mapsto \beta F(\alpha, \beta) \text{ is in }L^2(\R^2, d\mu(\alpha, \beta)) \big\}.
	\end{align*}
	\begin{definition}\label{def:phase_uncertainty}
	    The \emph{phase space uncertainty} associated to the window $f$ is defined to be
	    $$
	    \mathcal{L}_{\rm P}(f) = v_{\frac{K_{f}}{\|K_f\|_{2}}}(A) + v_{\frac{K_{f}}{\|K_f\|_{2}}}(B),
	    $$
	    where $\frac{K_{f}}{\|K_f\|_{2}} = W_{\frac{f}{\Vert f \Vert_\mathcal{W}}}\big[\frac{f}{\Vert f \Vert_\mathcal{S}}\big]$ is the normalized ambiguity function, and the variance is taken in the space $L^2(\R^2, d\mu(\alpha, \beta))$. The domain $\operatorname{Dom}(\mathcal{L}_{\rm P})$ is the set of all $f \in L^2(\R^+)$ such that $K_{f} \in \operatorname{Dom}(A) \cap \operatorname{Dom}(B)$.
	\end{definition}
	A main result in \cite{levie2021} is a pull-back of the calculation of the phase space uncertainty to the window function, based on the wavelet-Plancherel theory \cite{levie2017}, which makes it considerably easier to work with.
    \begin{proposition}[\cite{levie2021}]\label{prop:phase_uncer_pullback}
        Let $\mathcal{D}_{\rm P}$ be the set of $f \in L^2(\R^+)$ such that  $f$ is absolutely continuous in every compact interval,   $\Vert f \Vert_\mathcal{S} = 1$, $e_f^\mathcal{S}(T_x) = e_f^\mathcal{S}(T_\sigma) = 0$,   and the functions 
        $$
        \omega \mapsto f'(\omega),\qquad \omega \mapsto \sqrt{\omega} f'(\omega),\qquad \omega \mapsto\frac{f(\omega)}{\omega},\qquad \omega \mapsto\ln(\omega)f(\omega)
        $$
        are square integrable. Then, $\mathcal{D}_{\rm P} \subset \operatorname{Dom}(\mathcal{L}_{\rm P})$ and for all $f \in \mathcal{D}_{\rm P}$,
        \begin{align}\label{eq:phase_def}
    		\mathcal{L}_{\rm P}(f) = \big\Vert T_x f \big\Vert_\mathcal{S}^2 + \big\Vert T_\sigma f\big\Vert_\mathcal{S}^2 + v_{\frac{f}{\Vert f \Vert_{\mathcal{W}}}}^{\mathcal{W}}\Big(i\omega \frac{\partial}{\partial \omega}\Big)\left\Vert \frac{f}{\omega} \right\Vert_\mathcal{S}^2 + v^{\mathcal{W}}_{\frac{f}{\Vert f \Vert_{\mathcal{W}}}}(-\ln(\omega)).
    	\end{align}
    \end{proposition}
    Formula (\ref{eq:phase_def}) is similar to the signal space uncertainty \eqref{eq:simplified_signal_uncertainty}, with two added terms. The constraint $e_f^\mathcal{S}(T_x) = e_f^\mathcal{S}(T_\sigma) = 0$ in $\mathcal{D}_{\rm P}$ is taken for its signal processing utility. It assures that $f$ is centered at time and scale $0$, so that $W_f[s](\alpha,\beta)$ can be interpreted as the content of $s$ at the time-scale $(\alpha,\beta)$. The following proposition is analogous to Proposition \ref{prop:dom_S_admissible}.
    \begin{proposition}
	    Any element of $\mathcal{D}_{\rm P}$ is admissible, i.e., $\mathcal{D}_{\rm P} \subset \mathcal{W}\cap\mathcal{S}$.
	\end{proposition}
	\begin{proof}
	This follows by the fact that $\mathcal{D}_{\rm P} \subset \operatorname{Dom}(\mathcal{L}_{\rm S})$. 
	\end{proof}
	
	\section{Existence of signal space uncertainty minimizers}\label{sec:reg_uncer}
	In this section, we prove   our   main result on the existence of minimizers of the signal space uncertainty (Definition \ref{def:signal_uncertainty}). 
	\begin{theorem}\label{thm:exist_min_DS}
		There exists a minimizer of $\mathcal{L}_{\rm S}$ in $\operatorname{Dom}(\mathcal{L}_{\rm S})$.
	\end{theorem}
	We first note that by Lemma \ref{lemma:signal_invariance}, we can restrict our attention to the set 
	\begin{align*}
		\mathcal{D}_{\rm S} &= \Big\{ f \in \operatorname{Dom}(\mathcal{L}_{\rm S}) : \Vert f \Vert_\mathcal{S} =1,\, e_f(T_x) = 0,\, e_f(T_\sigma) = 0 \Big\},
	\end{align*}
	where the uncertainty simplifies to \eqref{eq:simplified_signal_uncertainty}, i.e.,
	$$
	\mathcal{L}_{\rm S}(f) = \big\Vert T_x f \big\Vert_\mathcal{S}^2 +  \big\Vert T_\sigma f \big\Vert_\mathcal{S}^2
	$$
	The following proposition is the key to proving existence.
	\begin{proposition}\label{prop:min_in_comp}
		Let $(f_n)_{n} \subset \mathcal{D}_{\rm S}$ be a minimizing sequence of $\mathcal{L}_{\rm S}(f)$ in the sense that
		$$
		\lim_{n \to \infty} \mathcal{L}_{\rm S}(f_n) = \inf_{y \in \mathcal{D}_{\rm S}}\mathcal{L}_{\rm S}(y).
		$$
		Then, there exist $N \in \N$ and a compact subset $\mathcal{K}_{\rm S} \subset \mathcal{D}_{\rm S}$ such that $f_n \in \mathcal{K}_{\rm S}$ for $n > N$. 
	\end{proposition}
	
	In the above proposition, note that a minimizing sequence exists since $\mathcal{D}_{\rm S}$ is non-empty and $\mathcal{L}_{\rm S}(\mathcal{D}_{\rm S})$ consists of non-negative real numbers. In the following analysis, we fix a value  $K>0$ such that 
	\begin{equation}
	\label{eq:K}
	  K/2> \inf_{y \in \mathcal{D}_{\rm S}}\mathcal{L}_{\rm S}(y).
	\end{equation} 
	
	%Next, we specify the subset $\mathcal{K}_{\rm S} \subset \mathcal{D}_{\rm S}$. To do so, we first note that the domain $\operatorname{Dom}(\mathcal{L}_{\rm S})$ is non-empty, and thus 
	%
	%there is a function $g\in \operatorname{Dom}(\mathcal{L}_{\rm S})$ with finite uncertainty. For example, when $g$ is the normalized real valued tent function
	%\begin{align}\label{eq:example_tent_function}
	%    g(\omega) = \begin{cases}
    %    	\sqrt{\frac{3}{2}}\omega,\qquad &0 \leq \omega \leq 1,\\
    %    	\sqrt{\frac{3}{2}}\big(2-\omega\big),\qquad &1 < \omega\leq 2,\\
    %    	0,\qquad&\text{otherwise},
    %	\end{cases}
	%\end{align}
	%it is easy to check that $\mathcal{L}_{\rm S}(g) = K/2 <\infty$ for some $K>0$.
	
%	\textcolor{blue}{Do we want to denote $\mathcal{K}_{\rm S}(K)$, or, for us $\mathcal{K}_{\rm S}$ already has the notation $K$ in it, since it is the calligraphic version of $K$? For example, for $G/2>\inf_{y \in \mathcal{D}_{\rm S}}\mathcal{L}_{\rm S}(y)$, the set would be  denoted by $\mathcal{G}_{\rm S}$.}
	\begin{definition}\label{def:Ks}
    %	Let $K>0$ satisfy $K/2>\inf_{y \in \mathcal{D}_{\rm S}}\mathcal{L}_{\rm S}(y)$. 
    	We define the subset  $\mathcal{K}_{\rm S}\subset \operatorname{Dom}(\mathcal{L}_{\rm S})$ to be 
    	\begin{equation}\label{eq:K_def}
    		\begin{aligned}
    			\mathcal{K}_{\rm S} = \bigg\{ f \in \operatorname{Dom}(\mathcal{L}_{\rm S}) :\, &\Vert f \Vert_\mathcal{S} = 1,\, e_f(T_x) = 0,\, e_f(T_\sigma) = 0,\,\\
    			&\hspace{30mm}\big\Vert T_x f \big\Vert_\mathcal{S}^2 \leq K,\, \big\Vert T_\sigma f \big\Vert_\mathcal{S}^2 \leq K \bigg\}
    		\end{aligned}
    	\end{equation}
	\end{definition}

	The following lemma is now easy to verify.
	\begin{lemma}\label{lemma:tail_S_in_K}
	    Let $(f_n)_n \subset \mathcal{D}_{\rm S}$ be a minimizing sequence of $\mathcal{L}_{\rm S}$. Then, there exists an $N \in \N$ such that $f_n \in \mathcal{K}_{\rm S}$ for $n \geq N$.
	\end{lemma}
	\begin{proof}
	 By (\ref{eq:K}), there exists an $N \in \N$ such that $f_n<K$ for $n \geq N$.
	    In particular, this means that $f_n \in \mathcal{K}_{\rm S}$ for large enough $n$ since both terms of $\mathcal{L}_{\rm S}$ are non-negative.
	\end{proof}

	We prove that $\mathcal{K}_{\rm S}$ is compact by showing that it is both closed and pre-compact. For the closedness, we begin by stating two auxiliary lemmas, the proofs of which we leave to the reader.
	\begin{lemma}\label{lemma:xf2leqM_closed}
		For any $M>0$, the set 
		$$
		\Big\{ q \in L^2(\R) : \int_\R y^2|q(y)|^2\, dy \leq M \Big\}
		$$
		is closed in $L^2(\R)$.
	\end{lemma}
	%\begin{proof} 
%		Let $(q_n)_n$ be a sequence such that $\int_\R y^2|q(y)|^2\, dy \leq M$ and write $q = q_n + \delta_n$. We show that $\int_{-N}^N y^2 |q(y)|^2\, dy \leq M$ for all $N \in \N$ by noting that for all $n$,
%		\begin{align*}
%			\int_{-N}^N y^2|q(y)|^2 \,dy &= \int_{-N}^N y^2 |q_n(y)|^2 \,dy + 2\operatorname{Re}\int_{-N}^N y^2 q_n(y) \overline{\delta_n(y)} \,dy + \int_{-N}^N y^2 |\delta_n(y)|^2\, dy\\
%			&\leq K + 2\underbrace{\sqrt{\int_{-N}^N y^2 |q_n(y)|^2 \,dy}}_{< \sqrt{M}} \underbrace{\sqrt{N^2\int_{-N}^N|\delta_n(y)|^2 \,dy}}_{\xrightarrow[n \to \infty]{} 0} + N^2\underbrace{\int_{-N}^N|\delta_n(y)|^2 \,dy}_{\xrightarrow[n \to \infty]{} 0} \xrightarrow[n \to \infty]{} M.
%		\end{align*}
%		Letting $N \to \infty$, we obtain the desired inequality.
%	\end{proof}
	
	\begin{lemma}\label{lemma:xf2eq0_closed}
	    For any $M>0$, the set 
		$$
		\Big\{ q \in L^2(\R) : \int_\R y^2|q(y)|^2\, dy \leq M,\, \int_\R y|q(y)|^2\, dy = 0 \Big\}
		$$
		is closed in $L^2(\R)$.
	\end{lemma}

	We are now ready to prove that  each of the conditions in \eqref{eq:K_def} defining $\mathcal{K}_{\rm S}$ corresponds to  a closed subset.  
	
	\begin{lemma}\label{lemma:T_closed}
		For any $M \geq 0$, the following subsets are closed in $L^2(\R^+)$:
		\begin{align*}
		    A= \Big\{ f \in \operatorname{Dom}(T_x) : \big\Vert T_x f \big\Vert_\mathcal{S}^2 \leq M,\, e_f(T_x) = 0 \Big\}, \\
			B= \Big\{ f \in \operatorname{Dom}(T_\sigma) :  \big\Vert T_\sigma f \big\Vert_\mathcal{S}^2 \leq M,\, e_f(T_\sigma) = 0 \Big\}.
		\end{align*}
	\end{lemma}
	\begin{proof}
    	By Proposition \ref{prop:frequency_observables}, $H^2(\R)$ is invariant under   multiplication by $t$ for $\check{f}$ in the domain of $\mathcal{L}_{\rm S}$. Hence, we can express the restrictions in the time and scale spaces by
    	\begin{align*}
    		A &= \Big\{ f \in \operatorname{Dom}(T_x) : \int t^2 |\check{f}(t)|^2 \,dt \leq M,\, \int t|\check{f}(t)|^2 \,dt = 0\Big\},\\\
    		B &= \Big\{ f \in \operatorname{Dom}(T_\sigma) : \int \sigma^2 |\tilde{f}(\sigma)|^2 \,d\sigma \leq M,\, \int \sigma |\tilde{f}(\sigma)|^2 \,d\sigma = 0\Big\}.
    	\end{align*}
    	By   Lemma \ref{lemma:xf2eq0_closed}, we obtain  that both $A$ and $B$ are closed. 
	\end{proof}
	%\textcolor{blue}{Do we want to add the following setting in blue? Or does  $\mathcal{K}_{\rm S}$ already contain the assumption with $K$ implicitly by definition?}
	\begin{proposition}\label{prop:K_closed}
	%\textcolor{blue}{Let $K>0$ satisfy $K/2>\inf_{y \in \mathcal{D}_{\rm S}}\mathcal{L}_{\rm S}(y)$. Then,}
 The set $\mathcal{K}_{\rm S}$ is closed in $L^2(\R^+)$.
	\end{proposition}
	\begin{proof}
		By writing $\mathcal{K}_{\rm S}$ as an intersection of sets corresponding to the conditions in  \eqref{eq:K_def}, and noting that   these sets are closed by Lemma \ref{lemma:T_closed}, it follows that $\mathcal{K}_{\rm S}$ is closed too. 
	\end{proof}
	To establish that $\mathcal{K}_{\rm S}$  is  pre-compact, we show that this set  can be approximated by compact sets with arbitrary small error. 
	
	\begin{lemma}\label{lemma:approx_K_by_compact}
		For any $\varepsilon > 0$, there exists a compact subset $C_{a,b}$ of $L^2(\R^+)$ such that for any $f \in \mathcal{K}_{\rm S}$, there is a $y\in C_{a,b}$ such that
		$$
		\Vert f - y \Vert_\mathcal{S} < \varepsilon.
		$$
	\end{lemma}
	\begin{proof} 
		For $f \in \mathcal{K}_{\rm S}$, we consider  $|\tilde{f}(\sigma)|^2$ as a probability distribution with mean value $0$ and variance $\Vert T_\sigma f \Vert_\mathcal{S}^2 \leq K$. Applying Chebyshev's inequality to the associated random variable, we have that for any $\alpha > 0$,
		$$
		\int_{[e^{-\alpha}, e^{\alpha}]^c}|f(\omega)|^2 \,d\omega = \int_{[-\alpha, \alpha]^c} \big|\tilde{f}(\sigma)\big|^2 \,d\sigma  \leq \frac{K}{\alpha^2}.
		$$
		Now, fix $\varepsilon > 0$ and choose $\alpha$ so large that $\frac{K}{\alpha^2} < \varepsilon$. Then, by the above inequality with $a = e^{-\alpha}$ and $b = e^\alpha$, it holds that $\Vert f - f\big|_{[a, b]} \Vert < \varepsilon$. 
		
		We now show that $f\big|_{[a, b]}$ is contained in a compact subset for each $f \in \mathcal{K}_{\rm S}$.
		First, we note that for every $f \in \mathcal{K}_{\rm S}$, we have
		$$
		\int_a^b |f'(\omega)|^2 \,d\omega \leq \left\Vert f' \right\Vert_\mathcal{S}^2 = \big\Vert T_x f \big\Vert_\mathcal{S}^2 \leq K.
		$$
		Therefore, if we mirror $f\big|_{[a, b]}$ around $\omega = b$ to $[a, 2b-a]$ and let $f_e$ denote the absolutely continuous periodic extension of the resulting function, it will hold that
		\begin{align*}
			\big\Vert f_e'\big\Vert^2 \leq 2K &\implies \big\Vert f_e'\big\Vert^2 = \sum_n \frac{\pi}{(b-a)^2}n^2 |c_n|^2 \leq 2K\\
			&\implies |c_n| \leq \frac{\sqrt{2K}(b-a)}{\pi}\frac{1}{|n|},
		\end{align*}
		where $(c_n)_n$ are the Fourier coefficients of $f_e$. Next, we define 
		$$
		H_{a, b} = \Bigg\{\text{Periodic functions with period } 2(b-a)\text{ such that }|c_n| \leq \frac{\sqrt{2K}(b-a)}{\pi}\frac{1}{|n|} \Bigg\},
		$$
		and note that by Parseval's formula and Tychonoff's theorem, this set, known as a the Hilbert cube, is compact. Moreover, we obtain that  
		$$
		C_{a, b} = \Big\{q \in L^2(a, b) : \exists\, y \in H_{a, b} : q = y\big|_{[a, b]}\Big\}
		$$
		is compact. Indeed, let $(q_n)_n$ be a sequence in $C_{a, b}$. Then each $q_n$ can be mirrored and extended to yield a sequence in $({q_e}_n)_n \subset H_{a, b}$ which has a convergent subsequence. Restricting back to $(a,b)$, we obtain a convergent subsequence of $(q_n)_n$. 
		
		Finally, since   $f\big|_{[a, b]} \in C_{a, b}$, we obtain  the desired conclusion with $y = f\big|_{[a, b]}$. 

	\end{proof}
	
	\begin{proposition}\label{prop:K_precompact}
		The set $\mathcal{K}_{\rm S}$ is pre-compact. 
	\end{proposition}
	\begin{proof}
		Fix a sequence $(f_n)_n \subset \mathcal{K}_{\rm S}$. We  prove that $(f_n)_n$ has a convergent subsequence  by constructing a Cauchy subsequence. By Lemma \ref{lemma:approx_K_by_compact}, we can choose $a_m, b_m$ for each $m$ such that any function in $\mathcal{K}_{\rm S}$ can be approximated by a function in $C_{a_m, b_m}$ with error less than $\frac{1}{m}$. For each $f_n$, we let $f_n^m$ denote these approximations. That is, 
		\begin{align}\label{eq:approx_of_K_by_C}
			f_n^m \in C_{a_m, b_m},\qquad \big\Vert f_n - f_n^m \big\Vert_\mathcal{S} < \frac{1}{m}.
		\end{align}
		For fixed $m$, the sequence $(f_n^m)_n$ is in the compact set $C_{a_m, b_m}$. There is thus a convergent subsequence
		$$
		(f^m_{n_j^m})_j \subset C_{a_m, b_m},\qquad f^m_{n_j^m} \xrightarrow[j \to \infty]{} f^m\,\text{ in }C_{a_m, b_m}.
		$$
		We choose   subsequences so that $(n_j^m)_j$ is a subsequence of $(n_j^{m'})_j$ for every $m > m'$. 
		
		Keeping $m$ fixed, we have, by \eqref{eq:approx_of_K_by_C},  that
		$$
		\Vert f_{n_j^m} - f_{n_j^m}^m \Vert_\mathcal{S} < \frac{1}{m}.
		$$
		Now, since $f^m_{n_j^m} \xrightarrow[j \to \infty]{} f^m$ in $C_{a_m, b_m}$, we can choose $j_m$ so large that
		$$
		\big\Vert f^m_{n_{j}^m} - f^m \big\Vert_\mathcal{S} < \frac{1}{m} \implies \big\Vert f_{n_j^m} - f^m\big\Vert_\mathcal{S} < \frac{2}{m}
		$$
		for all $j \geq j_m$. This implies that $(f_{n_{j_m}^m})_m$ is a Cauchy sequence. Indeed,  for every $m > m'$ it holds that
		$$
		\big\Vert f_{n_{j_m}^m} - f^{m'} \big\Vert_\mathcal{S} < \frac{2}{m'} \implies \big\Vert f_{n_{j_{m_1}}^{m_1}} - f_{n_{j_{m_2}}^{m_2}} \big\Vert_\mathcal{S} < \frac{4}{m'}
		$$
		for every $m_1, m_2 > m'$. By the completeness of $L^2(\R^+)$, the proof is complete.
	\end{proof}
We can now prove Proposition \ref{prop:min_in_comp} and Theorem \ref{thm:exist_min_DS}.
	\begin{proof}[Proof of Proposition \ref{prop:min_in_comp}]
		The tail of any minimizing sequence is in $\mathcal{K}_{\rm S}$ by Lemma \ref{lemma:tail_S_in_K}. The compactness of $\mathcal{K}_{\rm S}$ follows from Proposition \ref{prop:K_closed} ($\mathcal{K}_{\rm S}$ closed)  and Proposition \ref{prop:K_precompact} ($\mathcal{K}_{\rm S}$ pre-compact). 
	\end{proof}
	
	\begin{proof}[Proof of Theorem \ref{thm:exist_min_DS}]\label{proof:thm_exist_min_DS}
		By Proposition \ref{prop:min_in_comp},  there is a minimizing sequence in the compact subset $\mathcal{K}_{\rm S} \subset \mathcal{D}_{\rm S}$. This sequence, therefore,  converges to some point $f_0 \in \mathcal{K}_{\rm S}$.  Moreover, by the compactness of $[0, K]$, we can pass to a subsequence such that 
		$$
		\big\Vert T_x f_n \big\Vert_\mathcal{S}^2 \xrightarrow[n \to \infty]{} \sigma_1,\qquad \big\Vert T_\sigma f_n \big\Vert_\mathcal{S}^2 \xrightarrow[n \to \infty]{} \sigma_2,
		$$
		for some $\sigma_1, \sigma_2 \leq K$ such that $\sigma_1 + \sigma_2 = \inf_{y \in \mathcal{D}_{\rm S}}\mathcal{L}_{\rm S}(y)$.
		
		Next we show that $f_0$ is a minimizer of $\mathcal{L}_{\rm S}$. Define for every $\alpha,\beta>0$
		$$
		\mathcal{K}_{\alpha, \beta} = \Big\{ q \in \mathcal{K}_{\rm S} : \big\Vert T_x q \big\Vert_\mathcal{S}^2 \leq \alpha,\, \big\Vert T_\sigma q \big\Vert_\mathcal{S}^2 \leq \beta \Big\}
		$$
		and note that this set  is closed by Lemma \ref{lemma:T_closed}. Since $\mathcal{K}_{\alpha, \beta} \subset \mathcal{K}_{\rm S}$, this implies  that $\mathcal{K}_{\alpha, \beta}$ is compact.  Now, for each $\varepsilon > 0$, it holds that the tail of the minimizing sequence $(f_n)_n$ is contained in $\mathcal{K}_{\sigma_1 + \varepsilon, \sigma_2 + \varepsilon}$. Thus, for any $\epsilon>0$, we have
		$
		\mathcal{L}_{\rm S}(f_0) \leq \sigma_1 + \sigma_2 + 2\varepsilon$, which implies that $\mathcal{L}_{\rm S}(f_0) = \sigma_1 + \sigma_2
		$.
		Thus,  $f_0$ is a minimizer.
	\end{proof}
	
	\section{Existence of phase space uncertainty minimizers}\label{sec:phase_uncer}
	In this section, we prove our main result on the   existence of minimizers of the   phase space uncertainty (Definition \ref{def:phase_uncertainty}). 
	For the convenience of the reader, we recall that from Proposition \ref{prop:phase_uncer_pullback}, for $f \in \mathcal{D}_{\rm P}$, the   phase space uncertainty \eqref{eq:phase_def} can be expressed as %\olsen{slightly rewritten (should one recall for the reader also what $D_P$ is? If not, the reader anyway has to look back to where this formula \eqref{eq:phase_uncer_full} is stated).. another option is to not recall anything, which may be reasonable since the paper is much shorter than the thesis... especially if we skip the two proofs above, as I suggest.}\simon{I think, especially with the two proofs skipped, that it is better not to remind the reader of the definintion of $D_P$ here. Mainly because we never leave $D_P$ again and therefore don't care about its "boundaries"}
	\begin{align}\label{eq:phase_uncer_full}
		\mathcal{L}_{\rm P}(f) = \big\Vert T_x f\big\Vert_\mathcal{S}^2 + \big\Vert T_\sigma f\big\Vert_\mathcal{S}^2 + v_{\frac{f}{\Vert f \Vert_{\mathcal{W}}}}^{\mathcal{W}}\Big(i\omega \frac{\partial}{\partial \omega}\Big)\left\Vert \frac{f}{\omega} \right\Vert_\mathcal{S}^2 + v^{\mathcal{W}}_{\frac{f}{\Vert f \Vert_{\mathcal{W}}}}(-\ln(\omega)),
	\end{align}
	 where $\mathcal{D}_{\rm P}$ is the set of $f \in L^2(\R^+)$ such that  $f$ is absolutely continuous in every compact interval,   $\Vert f \Vert_\mathcal{S} = 1$, $e_f^\mathcal{S}(T_x) = e_f^\mathcal{S}(T_\sigma) = 0$,   and the functions 
        $$
        \omega \mapsto f'(\omega),\qquad \omega \mapsto \sqrt{\omega} f'(\omega),\qquad \omega \mapsto\frac{f(\omega)}{\omega},\qquad \omega \mapsto\ln(\omega)f(\omega)
        $$
        are square integrable.
	
	\begin{theorem}\label{thm:exist_min_Dl}
		There exists a minimizer of $\mathcal{L}_{\rm P}$ in $\mathcal{D}_{\rm P}$.
	\end{theorem}
	The proof follows a similar path to that of Theorem \ref{thm:exist_min_DS}. In particular,  the following proposition, which is analogous to Proposition \ref{prop:min_in_comp}, is a key step. 
	\begin{proposition}\label{prop:phase_min_in_comp}
		Let $(f_n)_{n} \subset \mathcal{D}_{\rm P}$ be a minimizing sequence of $\mathcal{L}_{\rm P}(f)$ in the sense that
		$$
		\lim_{n \to \infty} \mathcal{L}_{\rm P}(f_n) = \inf_{y \in \mathcal{D}_{\rm P}}\mathcal{L}_{\rm P}(y).
		$$
		Then, there exist $N \in \N$ and a compact subset $\mathcal{K}_{\rm P} \subset \mathcal{D}_{\rm P}$ such that $f_n \in \mathcal{K}_{\rm P}$ for $n > N$. 
	\end{proposition}
	Note that when the signal-expected time $e_f^\mathcal{S}(T_x)$ is zero, so is $e_f^\mathcal{W}\big(i \omega \frac{\partial}{\partial \omega}\big)$ since
	$$
	e_{\frac{f}{\Vert f \Vert_{\mathcal{W}}}}^\mathcal{W}\Big(i \omega \frac{\partial}{\partial \omega}\Big) = \frac{1}{\Vert f \Vert_{\mathcal{W}}^2} e_f\Big(i\frac{\partial}{\partial \omega}\Big).
	$$
	We can therefore further simplify the uncertainty \eqref{eq:phase_uncer_full} for $f \in \mathcal{D}_{\rm P}$ to
	\begin{align*}\label{eq:phase_uncer_simplified}
		\mathcal{L}_{\rm P}(f) = \big\Vert T_x f\big\Vert_\mathcal{S}^2 + \big\Vert T_\sigma f \big\Vert_\mathcal{S}^2 + \frac{1}{\Vert f \Vert_{\mathcal{W}}^2} \big\Vert i\omega f' \big\Vert_{\mathcal{W}}^2\left\Vert \frac{f}{\omega} \right\Vert_\mathcal{S}^2 + \frac{1}{\Vert f \Vert_{\mathcal{W}}^2} v_f^\mathcal{W}(-\ln(\omega)).
	\end{align*}
	
	From here, we   proceed as in Section \ref{sec:reg_uncer}. First, we  define $\mathcal{K}_{\rm P}$ analogously to Definition \ref{def:Ks}. We fix a value  $K>0$ such that 
	\begin{equation}
	\label{eq:K2}
	  K/2> \inf_{y \in \mathcal{D}_{\rm S}}\mathcal{L}_{\rm P}(y).
	\end{equation} 
	\begin{definition}
    	 The domain $\mathcal{K}_{\rm P}\subset \mathcal{D}_{\rm P}$ is defined to be 
    	\begin{equation}\label{eq:K_large_intersection}
    		\begin{aligned}
    			\mathcal{K}_{\rm P} = \Bigg\{ f \in \operatorname{Dom}(\mathcal{L}_{\rm P}) :\, &\Vert f \Vert_\mathcal{S} = 1,\, e_f(T_x) = 0,\, e_f(T_\sigma) = 0,\\
    			&\big\Vert T_x f \big\Vert_\mathcal{S}^2 \leq K,\, \frac{1}{\Vert f \Vert_{\mathcal{W}}^2} \big\Vert i\omega f' \big\Vert_{\mathcal{W}}^2\left\Vert \frac{f}{\omega} \right\Vert_\mathcal{S}^2 \leq K,\\[2mm]
    			&\hspace{30mm}\big\Vert T_\sigma f \big\Vert_\mathcal{S}^2 \leq K,\, \frac{1}{\Vert f \Vert_{\mathcal{W}}^2} v_{f}^\mathcal{W}(-\ln(\omega)) \leq K \Bigg\}.
    		\end{aligned}
    	\end{equation}
	\end{definition}
	The following lemma follows by the same argument as for Lemma \ref{lemma:tail_S_in_K}. 
	\begin{lemma}\label{lemma:phase_tail_compact}
	    Let $(f_n)_n \subset \mathcal{D}_{\rm P}$ be a minimizing sequence of $\mathcal{L}_{\rm P}$. Then there exists an $N \in \N$ such that $f_n \in \mathcal{K}_{\rm P}$ for $n \geq N$.
	\end{lemma}
	
	To prove Proposition \ref{prop:phase_min_in_comp}, we show that $\mathcal{K}_{\rm P}$ is compact. Since $\mathcal{K}_{\rm P} \subset \mathcal{K}_{\rm S}$ and $\mathcal{K}_{\rm S}$ is compact, it only remains to show that $\mathcal{K}_{\rm P}$ is closed.  Lemma \ref{lemma:T_closed} already shows that two of the conditions in \eqref{eq:K_large_intersection} correspond to closed sets. For the remaining two conditions, we   need the following auxiliary lemmas. 
	\begin{lemma}\label{lemma:CS_f}
		Let $M > 0$ and $f \in \operatorname{Dom}(T_x)$ be such that $\big\Vert T_x f \big\Vert_\mathcal{S}^2 \leq M$ and $f(0) = 0$. Then, 
		$$
		|f(\omega)| \leq \sqrt{M}\sqrt{\omega}.
		$$
	\end{lemma}
	\begin{proof}
	    This follows from an application of the   Cauchy-Schwarz inequality. 
	\end{proof}
	\begin{lemma}\label{lemma:middle_K_estimate}
		Let $M > 0$ and $f \in L^2(\R^+)$ be such that $\Vert T_x f \Vert_\mathcal{S}^2 \leq M$ and $\frac{1}{\Vert f \Vert_{\mathcal{W}}^2} \big\Vert \omega f' \big\Vert_{\mathcal{W}}^2\left\Vert \frac{f}{\omega} \right\Vert_\mathcal{S}^2 \leq M$. Then, 
		$$
		\big\Vert \omega f' \big\Vert_{\mathcal{W}}^2 \leq 2 e^2 M(M+1).
		$$
	\end{lemma}
	\begin{proof}
	    Note first that $\Vert \omega f' \Vert_{\mathcal{W}}^2$ is the second factor of $\frac{1}{\Vert f \Vert_{\mathcal{W}}^2} \big\Vert \omega f' \big\Vert_{\mathcal{W}}^2\left\Vert \frac{f}{\omega} \right\Vert^2 \leq M$. Therefore, by bounding $\frac{1}{\Vert f \Vert_{\mathcal{W}}^2}$ and $\left\Vert \frac{f}{\omega} \right\Vert^2$ from below, we obtain a bound of $\Vert \omega f' \Vert_{\mathcal{W}}^2$ from above. For $\frac{1}{\Vert f \Vert_\mathcal{W}^2}$, we bound
		$$
		\Vert f \Vert_\mathcal{W}^2 = \int_0^\infty \frac{|f(\omega)|^2}{\omega} d\omega \leq \int_0^1 \frac{|f(\omega)|^2}{\omega} d\omega + \underbrace{\Vert f \Vert^2}_{=1} \leq \int_0^1\frac{M \omega}{\omega} d\omega + 1 = M+1,
		$$
		where we made use of Lemma \ref{lemma:CS_f}. Hence,
		\begin{align*}
			\frac{1}{\Vert f \Vert_{\mathcal{W}}^2} \big\Vert \omega f' \big\Vert_{\mathcal{W}}^2\left\Vert \frac{f}{\omega} \right\Vert_\mathcal{S}^2 \leq M \implies \big\Vert \omega f' \big\Vert_{\mathcal{W}}^2\left\Vert \frac{f}{\omega} \right\Vert_\mathcal{S}^2 \leq \frac{M}{\frac{1}{\Vert f \Vert_{\mathcal{W}}^2}} \leq M(M+1).
		\end{align*}
		
		To bound $\left\Vert \frac{f}{\omega} \right\Vert_\mathcal{S}^2$ from below, we consider the following two cases separately.\\
		\noindent
		\textbf{Case 1: }$\int_e^\infty |f(\omega)|^2 d\omega \geq 1/2$.\\
		By the fact that that $1/\omega^2 > |\ln(\omega)|$ in $(0,1)$ and by the fact that
		$$
		e_f^\mathcal{S}(T_\sigma) = 0 \implies \int_0^1 |\ln(\omega)| |f(\omega)|^2 d\omega = \int_1^\infty |\ln(\omega)| |f(\omega)|^2 d\omega,
		$$
		we have
		\begin{align*}
			\left\Vert \frac{f}{\omega} \right\Vert_\mathcal{S}^2 &= \int_0^\infty \frac{|f(\omega)|^2}{\omega^2} d\omega \geq \int_0^1 \frac{|f(\omega)|^2}{\omega^2} d\omega \geq \int_0^1 |\ln(\omega)| |f(\omega)|^2 d\omega\\[2mm]
			&= \int_1^\infty |\ln(\omega)| |f(\omega)|^2 d\omega \geq \int_e^\infty \underbrace{|\ln(\omega)|}_{\geq 1} |f(\omega)|^2 d\omega \geq \int_e^\infty |f(\omega)|^2 d\omega \geq 1/2.
		\end{align*}
		\noindent
		\textbf{Case 2: }$\int_0^e |f(\omega)|^2 d\omega \geq 1/2$.\\
		In this case, we use the estimate
		$$
		\left\Vert \frac{f}{\omega} \right\Vert_\mathcal{S}^2 = \int_0^\infty \frac{|f(\omega)|^2}{\omega^2} d\omega \geq \int_0^e \frac{|f(\omega)|^2}{\omega^2} d\omega \geq \int_0^e \frac{|f(\omega)|^2}{e^2} d\omega \geq \frac{1}{2e^2}.
		$$
		Hence,  $\left\Vert \frac{f}{\omega} \right\Vert_\mathcal{S}^2 \geq \frac{1}{2e^2}$ uniformly. As a result,
		$$
		\big\Vert \omega f' \big\Vert_\mathcal{W}^2 \leq \frac{M (M+1)}{\left\Vert \frac{f}{\omega} \right\Vert_\mathcal{S}^2} \leq 2 e^2 M(M+1),
		$$
		which completes the proof.
	\end{proof}
	\begin{lemma}\label{lemma:W_cont}
		Let $M > 0$ and $(f_n)_n$ be a sequence in $L^2(\R^+)$ with $\Vert T_x f_n \Vert_\mathcal{S}^2 \leq M$ for each $n$ such that $\Vert f_n - f \Vert_\mathcal{S} \to 0$ for some $f \in L^2(\R^+)$. Then $\Vert f_n - f\Vert_\mathcal{W} \to 0$.
	\end{lemma}
	\begin{proof}
		Fix $\varepsilon > 0$ and choose $n$ so large that $\Vert f_n - f \Vert_\mathcal{S}^2 < \varepsilon^2/M$. Then
		\begin{align*}
			\Vert f_n - f \Vert_{\mathcal{W}}^2 &= \int_{0}^{\varepsilon / M} \frac{|f_n(\omega) - f(\omega)|^2}{\omega} d\omega + \int_{\varepsilon / M}^{\infty} \frac{|f_n(\omega) - f(\omega)|^2}{\omega} d\omega\\
			&< 4M \frac{\varepsilon}{M} + \frac{M}{\varepsilon} \Vert f_n - f \Vert_\mathcal{S}^2 < 5 \varepsilon,
		\end{align*}
		where we used Lemma \ref{lemma:CS_f} for the first term.
	\end{proof}
	We are now ready to prove two lemmas corresponding to the remaining two conditions of \eqref{eq:K_large_intersection}.
		\begin{lemma}\label{lemma:WS_triple_leqK}
		For any $M > 0$, the set
		$$
		\Bigg\{ f \in L^2(\R^+) : \frac{1}{\Vert f \Vert_{\mathcal{W}}^2} \big\Vert i\omega f' \big\Vert_{\mathcal{W}}^2\left\Vert \frac{f}{\omega} \right\Vert_\mathcal{S}^2 \leq M,\, \Vert T_x f\Vert_\mathcal{S}^2 \leq M \Bigg\}
		$$
		is closed in $L^2(\R^+)$.
	\end{lemma}
	\begin{proof}
		By writing
		$$
		A(f) = \big\Vert i\omega f'\big\Vert_\mathcal{W}^2,\qquad B(f) = \left\Vert \frac{f}{\omega}\right\Vert_\mathcal{S}^2,
		$$
		we have that
		$$
		\frac{1}{\Vert f \Vert_{\mathcal{W}}^2} \big\Vert i\omega f' \big\Vert_{\mathcal{W}}^2\left\Vert \frac{f}{\omega} \right\Vert_\mathcal{S}^2 \leq M \iff A(f)B(f) \leq M \Vert f \Vert_\mathcal{W}^2.
		$$
		Using the change of variables $\omega = e^s$, we obtain
		\begin{align*}
			A(f) = \int_0^\infty \omega \big|f'(\omega)\big|^2 d\omega = \int_{-\infty}^\infty e^s \big|f'(e^s)\big|^2 e^s ds = \int_\R \big|e^s f'(e^s)\big|^2 ds = \int_\R \big|f^{\flat\prime}(s)\big|^2 ds
		\end{align*}
		where $f^\flat(s) = f(e^s)$. It can be verified that $\Vert f^\flat \Vert_\mathcal{S} = \Vert f \Vert_{\mathcal{W}}$, which allows us to   apply the Fourier transform to $f^\flat$ to get 
		$$
		A(f) = 4 \pi^2 \int_\R t^2 \big|\widehat{f^\flat}(\xi)\big|^2 d\xi.
		$$
		
		Next, let $(f_n)$ be a sequence such that $\Vert f_n - f \Vert_\mathcal{S} \to 0$ for some $f \in L^2(\R^+)$ and write $\delta_n = f-f_n$. Then, it holds that 
		\begin{align*}
			\int_{\R} |\widehat{\delta_n^\flat}(t)|^2 dt &= \int_{\R} |\delta_n^\flat(s)|^2 ds = \int_\R |\delta_n(e^s)|^2 ds\\
			&= \int_0^\infty |\delta_n(\omega)|^2 \frac{1}{\omega} d\omega = \Vert f_n - f \Vert_\mathcal{W}^2 \xrightarrow[n \to \infty]{} 0,
		\end{align*}
		where we used the Plancherel theorem for the first step, and Lemma \ref{lemma:W_cont} for the last step. 
		
		We now estimate
		\begin{align*}
			&\left(4 \pi^2\int_{-N}^{N} \xi^2 |\widehat{f^\flat}(\xi)|^2 d\xi \right)\left(\int_{1/N}^{N}\frac{1}{\omega^2}|f(\omega)|^2 \,d\omega\right)\\
			&\leq \vast( \underbrace{4 \pi^2 \int_{-N}^{N} \xi^2  \big|\widehat{f_n^\flat}(\xi)\big|^2 \,d\omega}_{\leq A(f_n)} + 2\cdot 4 \pi^2\sqrt{\underbrace{\int_{-N}^{N} \xi^2 \big|\widehat{f_n^\flat}(\xi)\big|^2 \,d\xi}_{\leq A(f_n) \leq M(M+1) 2 e^2}} \sqrt{N^2\int_{-N}^{N} \big|\widehat{\delta_n^\flat}(\xi)\big|^2 \,d\xi}\\
			&\,\,\,\,+ 4 \pi^2 N^2 \int_{-N}^{N} \big|\widehat{\delta_n^\flat}(\xi)\big|^2 \,d\xi \vast)  \vast( \underbrace{\int_{1/N}^{N} \frac{1}{\omega^2} |f_n(\omega)|^2 \,d\omega}_{\leq B(f_n)}\\
			&\,\,\,\,+ 2\sqrt{N^2 \underbrace{\int_{1/N}^{N} |f_n(\omega)|^2 \,d\omega}_{\leq \Vert f_n \Vert^2 \leq 1}} \sqrt{N^2 \int_{1/N}^{N} |\delta_n(\omega)|^2 \,d\omega} + N^2\int_{1/N}^{N} |\delta_n(\omega)|^2 \,d\omega \vast)\\
			&= A(f_n)B(f_n) + o_n(1) \leq M \Vert f_n \Vert_{\mathcal{W}}^2 + o_n(1)
		\end{align*}
		%\begin{align*}
		%	&\left(4 \pi^2\int_{-N}^{N} \xi^2 |\widehat{f^\flat}(\xi)|^2 d\xi \right)\left(\int_{1/N}^{N}\frac{1}{\omega^2}|f(\omega)|^2 \,d\omega\right)\\
		%	&\leq \left( \underbrace{4 \pi^2 \int_{-N}^{N} \xi^2  \big|\widehat{f_n^\flat}(\xi)\big|^2 \,d\omega}_{\leq A(f_n)} + 2\cdot 4 \pi^2\sqrt{\underbrace{\int_{-N}^{N} \xi^2 \big|\widehat{f_n^\flat}(\xi)\big|^2 \,d\xi}_{\leq A(f_n) \leq M(M+1) 2 e^2}} \sqrt{N^2\int_{-N}^{N} \big|\widehat{\delta_n^\flat}(\xi)\big|^2 \,d\xi} + 4 \pi^2 N^2 \int_{-N}^{N} \big|\widehat{\delta_n^\flat}(\xi)\big|^2 \,d\xi\right)\\
		%	&\,\,\,\,\times \left( \underbrace{\int_{1/N}^{N} \frac{1}{\omega^2} |f_n(\omega)|^2 \,d\omega}_{\leq B(f_n)} + 2\sqrt{N^2 \underbrace{\int_{1/N}^{N} |f_n(\omega)|^2 \,d\omega}_{\leq \Vert f_n \Vert^2 \leq 1}} \sqrt{N^2 \int_{1/N}^{N} |\delta_n(\omega)|^2 \,d\omega} + N^2\int_{1/N}^{N} |\delta_n(\omega)|^2 \,d\omega\right)\\
		%	&= A(f_n)B(f_n) + o_n(1) \leq M \Vert f_n \Vert_{\mathcal{W}}^2 + o_n(1)
		%\end{align*}
		where we in the last step used the fact that $\Vert \delta_n \Vert_\mathcal{S} \to 0$ and $\big\Vert \widehat{\delta_n^\flat} \big\Vert_\mathcal{S} \to 0$ as $n \to \infty$. 
		
		Finally, since $\Vert f_n \Vert_\mathcal{W} \to \Vert f \Vert_{\mathcal{W}}$, by Lemma \ref{lemma:W_cont} and by letting $n \to \infty$, we conclude, for all $N$, that 
		\begin{align}\label{eq:AB_finite_estimate}
		    \left(\int_{-N}^{N} t^2 \big|\widehat{f^\flat}(\xi)\big|^2 d\xi \right)\left(\int_{1/N}^{N}\frac{1}{\omega^2}|f(\omega)|^2 \,d\omega\right) \leq M \Vert f \Vert_{\mathcal{W}}^2.
		\end{align}
		 The left-hand side of \eqref{eq:AB_finite_estimate} converges to $A(f)B(f)$ as $N \to \infty$, which yields
		$$
		A(f)B(f) = \big\Vert i\omega f' \big\Vert_{\mathcal{W}}^2\left\Vert \frac{f}{\omega} \right\Vert_\mathcal{S}^2 \leq M \Vert f \Vert_{\mathcal{W}}^2.
		$$
		This concludes the proof.
	\end{proof}
	
	\begin{lemma}\label{lemma:WS_var_scale_leq_K}
		For any $M > 0$, the set
		$$
		\Bigg\{f \in L^2(\R^+) : \frac{1}{\Vert f \Vert_{\mathcal{W}}^2} v_f^\mathcal{W}(-\ln(\omega)) \leq M,\, \big\Vert T_x f \big\Vert_\mathcal{S}^2 \leq M \Bigg\}
		$$
		is closed in $L^2(\R^+)$.
	\end{lemma}
	\begin{proof}
		We begin by writing the variance as
		$$
		v_f^\mathcal{W}(-\ln(\omega)) = e_f^\mathcal{W}(\ln(\omega)^2) - e_f^\mathcal{W}(-\ln(\omega))^2.
		$$
		As a result, $\frac{1}{\Vert f \Vert_{\mathcal{W}}^2} v_f^\mathcal{W}(-\ln(\omega)) \leq M$ can be written as
		\begin{equation}\label{eq:window_var_ineq_phase}
    		\begin{aligned}
    		    \int_{0}^{\infty}\frac{\ln(\omega)^2}{\omega} |f(\omega)|^2 \,d\omega &\leq M \Vert f \Vert_{\mathcal{W}}^2 + \left(\int_{0}^{\infty} \frac{-\ln(\omega)}{\omega}|f(\omega)|^2 \,d\omega\right)^2\\[2mm]
    		    &= M \Vert f \Vert_{\mathcal{W}}^2 + e_f^\mathcal{W}(-\ln(\omega))^2.
    		\end{aligned}
		\end{equation}
		To show that the inequality \eqref{eq:window_var_ineq_phase} corresponds to a closed subset, we let $(f_n)_n$ converge to $f$ in $L^2(\R^+)$, such that each $f_n$ satisfies \eqref{eq:window_var_ineq_phase},   and show that $f$ also satisfies \eqref{eq:window_var_ineq_phase}.  We first show that 
		\begin{align}\label{eq:conv_window_expected_scale}
			e_{f_n}^\mathcal{W}(-\ln(\omega)) \to e_{f}^\mathcal{W}(-\ln(\omega))
		\end{align}
		as $n \to \infty$. To see this, let $\varepsilon > 0$ be so small that $\varepsilon / M < 1/2$, and choose $n$ so large that $\big|\Vert f_n\Vert_\mathcal{S}^2 - \Vert f \Vert_\mathcal{S}^2 \big| < \varepsilon^2$. Then
		\begin{align*}
			\Big|e_{f_n}^\mathcal{W}(-\ln(\omega)) - e_{f}^\mathcal{W}(-\ln(\omega))\Big| &= \left|\int_{0}^{\infty} \frac{-\ln(\omega)}{\omega} \big(|f_n(\omega)|^2 - |f(\omega)|^2\big) d\omega\right|\\
			&\leq \int_{0}^{\varepsilon/M} |\ln(\omega)|2M d\omega + \int_{\varepsilon/M}^{\infty} \underbrace{\frac{|\ln(\omega)|}{\omega}}_{\leq \frac{\ln(\varepsilon/M)}{\varepsilon/M}}\big||f_n(\omega)|^2 - |f(\omega)|^2\big| d\omega\\
			&\leq \frac{\varepsilon}{M}\left(\ln\left(\varepsilon/M\right) - 1\right)2M + M\frac{\ln(\varepsilon/M)}{\varepsilon}\left| \Vert f_n \Vert_\mathcal{S}^2 - \Vert f \Vert_\mathcal{S}^2 \right|\\
			&\leq \varepsilon\left( 3 \ln(\varepsilon/M) - 2 \right).
		\end{align*}
		Next, with $\delta_n = f-f_n$, we have that for each $N$, 
		\begin{align*}
			\int_{1/N}^{N}\frac{\ln(\omega)^2}{\omega} |f(\omega)|^2 \,d\omega &= \int_{1/N}^{N} \frac{\ln(\omega)^2}{\omega} |f_n(\omega)|^2 \,d\omega + 2\operatorname{Re}\int_{1/N}^{N} \frac{\ln(\omega)^2}{\omega} f_n(\omega) \overline{\delta_n(\omega)} \,d\omega \\[2mm]
			&\hspace{10mm}\,\,\,+ \int_{1/N}^{N} \frac{\ln(\omega)^2}{\omega} |\delta_n(\omega)|^2 \,d\omega \\[2mm]
			&\leq M \Vert f_n \Vert_\mathcal{W}^2 + e_{f_n}^\mathcal{W}(-\ln(\omega))^2\\[2mm]
			&\hspace{10mm}\,\,\,+ 2\sqrt{\frac{\ln(1/N)^2}{N}\int_{1/N}^{N} |f_n(\omega)|^2 \,d\omega} \sqrt{\frac{\ln(1/N)^2}{N}\int_{1/N}^{N}|\delta_n(\omega)|^2 \,d\omega}\\[2mm]
			&\hspace{20mm}\,\,\,+ \frac{\ln(1/N)^2}{N}\int_{1/N}^{N}|\delta_n(\omega)|^2 \,d\omega.
		\end{align*}
	    By Lemma \ref{lemma:W_cont} and \eqref{eq:conv_window_expected_scale}, the first two terms of the upper bound, above, converge to $M \Vert f \Vert_\mathcal{W}^2 + e_f^\mathcal{W}(-\ln(\omega))^2$ as $n \to \infty$, while the remaining terms  all vanish as $n \to \infty$ since $\Vert \delta_n \Vert_\mathcal{S} \to 0$. We therefore have that 
		$$
		\int_0^\infty \frac{\ln(\omega)^2}{\omega}|f(\omega)|^2 \,d\omega = \lim_{N \to \infty} \int_{1/N}^N \frac{\ln(\omega)^2}{\omega}|f(\omega)|^2 \,d\omega \leq M \Vert f \Vert_{\mathcal{W}}^2 + e_f^\mathcal{W}(-\ln(\omega))^2,
		$$
		which concludes the proof.
	\end{proof}
	
	\begin{proposition}\label{prop:phase_K_closed}
		The set $\mathcal{K}_{\rm P}$ is closed in $L^2(\R^+)$.
	\end{proposition}
	\begin{proof}
		This follows directly from lemmas \ref{lemma:T_closed}, \ref{lemma:WS_triple_leqK} and \ref{lemma:WS_var_scale_leq_K} by taking the intersections of the sets appearing in these lemmas. 
	\end{proof}
	\begin{proof}[Proof of Proposition \ref{prop:phase_min_in_comp}]
	    By Lemma \ref{lemma:phase_tail_compact}, the tail of any minimizing sequence is in $\mathcal{K}_{\rm P}$. Moreover, $\mathcal{K}_{\rm P}$ is compact since, by Lemma \ref{prop:phase_K_closed},  it is a closed subset of $\mathcal{K}_{\rm S}$. 
	\end{proof}
	
	With Proposition \ref{prop:phase_min_in_comp} established, the proof of Theorem \ref{thm:exist_min_Dl} follows by the same procedure used to prove Theorem \ref{thm:exist_min_DS}. The only modifications are adjustments for equation \eqref{eq:phase_uncer_full} defining $\mathcal{L}_{\rm P}(f)$ consisting of four terms and replacing $\mathcal{D}_{\rm S}$ and $\mathcal{K}_{\rm S}$ by $\mathcal{D}_{\rm P}$ and $\mathcal{K}_{\rm P}$.

    \section{Conclusion}
    We have proven non-constructively the existence of minimizers for the signal space uncertainty  and phase space uncertainty functionals $\mathcal{L}_{\rm S}$ and $\mathcal{L}_{\rm P}$.  In \cite{levie2021}, a numerical gradient descent scheme was presented, which estimates a minimizer of the uncertainty functional $\mathcal{L}_{\rm P}$.  However, no approximation results were proven for the numerical scheme. In a future work, we  prove that discrete numerical estimates of wavelet uncertainty minimizers indeed approximate true minimizers in $L^2(\R^+)$, for both $\mathcal{L}_{\rm S}$ and $\mathcal{L}_{\rm P}$.

	\printbibliography
    %\bibliographystyle{plain}
    %\bibliography{ref.bib}
    
\end{document}